\documentclass[11pt]{article}
\usepackage{amsfonts,latexsym,rawfonts,amsmath,amssymb,amsthm,a4wide, times, color}
\usepackage{graphicx}
\usepackage[plainpages=false]{hyperref}
\numberwithin{equation}{section}
\usepackage{hyperref}
\usepackage[titletoc,title]{appendix}

 \usepackage[all]{xy}
 \usepackage{eufrak}
 \usepackage{makeidx}
 \usepackage{graphicx,psfrag}

\numberwithin{equation}{section}

\newcommand{\beq}{\begin{equation}}
\newcommand{\eeq}{\end{equation}}
\newcommand{\beqs}{\begin{eqnarray*}}
\newcommand{\eeqs}{\end{eqnarray*}}
\newcommand{\beqn}{\begin{eqnarray}}
\newcommand{\eeqn}{\end{eqnarray}}
\newcommand{\beqa}{\begin{array}}
\newcommand{\eeqa}{\end{array}}

\def\x{{\mathbf{x}}}

\def\n{{\mathbf{n}}}

\newtheorem{prop}{Proposition}[section]
\newtheorem{theo}[prop]{Theorem}
\newtheorem{lem}[prop]{Lemma}

\newtheorem{cor}[prop]{Corollary}

\title{ A local estimate for the mean curvature flow}
\author{Zhen Wang %first author
}

\begin{document}
\bibliographystyle{plain}

%\date{}

\maketitle

\begin{abstract}
We establish a pointwise estimate of $A$ along the mean curvature flow in terms of the initial geometry and the $|HA|$ bound. As corollaries we obtain the extension theorem of $HA$ and the blowup rate estimate of $HA$.
\end{abstract}

\tableofcontents

\section{Introduction}
Let $\x_0:\Sigma^n\to\mathbb{R}^{n+1}$ be a complete smooth immersed hypersurface without boundary and a family of immersions $\x(x,t):\Sigma^n\times[0,T)\to\mathbb{R}^{n+1}$ be a solution to the equation 
$$\partial_{t}\x=-H\n,\quad \x(0)=\x_0,$$
which is called a mean curvature flow with the maximal time $T\leq\infty$. When $\Sigma$ is compact, $T<\infty$ by the comparison principle for mean curvature flow and $\sup_{\Sigma_t}|A|\to\infty$ as $t\to T$ by Huisken in \cite{[H]}.

Since the finite-time singularity for a compact mean curvature flow is characterized by the blowup of the second fundamental form, it is of great interest to express this criterion in terms of some simpler quantity. Let $n\geq2$. A natural conjecture is the blowup of the mean curvature $H$, which was confirmed in dimension two by Li-Wang \cite{[LW]}. In \cite{[C]} Cooper proved the $HA$ tensor blows up at time $T$. In \cite{[C],[LS],[LS3]}, Cooper and Le-Sesum proved that the mean curvature blows up assuming some blowup rate of the second fundamental form. Some extension results under integral conditions can be seen in Le-Sesum \cite{[LS2]} and Xu-Ye-Zhao \cite{[XYZ]}. 

Similar blowup problems and extension problems have been studied for Ricci flow as well. In \cite{[Ha]} Hamilton proved that the Riemann curvature tensor blows up at the finite singular time. In \cite{[S]} Sesum proved the blowup of the Ricci curvature. In \cite{[W2],[CW],[KMW]}, Wang, Chen-Wang and Kotschwar-Munteanu-Wang arrived at estiamtes on curvature growth in terms of the Ricci curvature. 

The explict local estimate in Kotschwar-Munteanu-Wang \cite{[KMW]} has some precedent on a gradient shrinking soliton in \cite{[MW]}, a bound on $Ric$ implies a polynomial growth bound on $Rm$. The feasibility lies in the observation that the second order derivatives of $Ric$ appear as time-derivative of $Rm$, i.e., 
$$\partial_{t}Rm=c\nabla^2 Ric,$$
which helps to yield a differential inequality on spatial integrations. This equation follows from the fact that $Ric$ describes the metric evolution along a Ricci flow. Now, for a mean curvature flow $HA$ describes the metric evolution and plays the role of $Ric$ by
$$\partial_{t}|A|^2=2(\nabla^{2}H\cdot A+ H|A^3|).$$
In the present paper, we follow the techniques on integration estimates from \cite{[KMW]} and establish the following local $L^\infty$ estimate of $A$ in terms of the initial geometry and the $|HA|$ bound along the flow.

\begin{theo}(Theorem \ref{theo:local est})
    Fix $x_0\in\mathbb{R}^{n+1}$ and $r>0$. Let $\x:\Sigma^n\times[t_0,t_1]\rightarrow \mathbb{R}^{n+1}$ be a complete smooth mean curvature flow satisfying the uniform bound 
    $$\sup_{B(x_0,r)\cap \Sigma_t}|HA|(\cdot,t)\leq K(t),\quad \forall\, t\in[t_0,t_1].$$
    Then for any $q>n+2$ there exist positive constants $C=C(n,r,t_1-t_0,q,K(t_0))$ and $c=c(n,q)$ such that for any $t\in[t_0,t_1]$
    \beqs
    \sup_{B(x_0,r/2)\cap \Sigma_t}|A|
    \leq C\Big(1+\|A\|_{L^{q}(B_{2r})}^{q}\Big)^c\Big(1+\mbox{Vol}_{g(t_0)}(B_{2r})\Big)^c
    \Big(\int_{t_0}^{t}e^{\int_{t_0}^{s}(|K'|/K+cK)}ds\Big)^{c},
    \eeqs
    where $B_{2r}=B(x_0,2r+n^{1/4}\int_{t_0}^{t}\sqrt{K})\cap \Sigma_{t_0}.$
\end{theo}

This result provides a new proof of the blowup of $HA$ in \cite{[C]} and extends the estimates for the Ricci flow in terms of $Ric$ in \cite{[W2],[CW],[KMW]} to an estimate for the mean curvature flow in terms of $HA$. One of its direct corollaries is the following extension theorem in terms of $HA$.

\begin{theo}(Corollary \ref{cor:AH-extension})
    Let $\x:\Sigma^n\times[0,T)\rightarrow \mathbb{R}^{n+1}$ be a complete noncompact smooth mean curvature flow. Each time slice $\Sigma_t$ has bounded $HA$.
     There exists a positive constant $C=C(n,T,K,V,E,q)$ such that if
     \begin{enumerate}
\item[(1)] the $|HA|$ bound satisfies $$\sup_{t\in[0,T)}\sup_{\Sigma_t}|HA|(\cdot,t)\leq K<\infty;$$  
\item[(2)] the initial data satisfies a uniform volume bound 
\beqs
\sup_{x\in\Sigma_0}\mbox{Vol}_{g(0)}(B(x,1+n^{1/4}T\sqrt{K})\cap\Sigma_0)\leq V<\infty;
\eeqs
\item[(3)] the initial data satisfies an integral bound 
\beqs
\sup_{x\in\Sigma_0}\|A\|_{L^{q}(B(x,1)\cap\Sigma_0)}\leq E<\infty
\eeqs
for some $q>n+2$,
     \end{enumerate}
     then
     $$\limsup_{t\to T}\sup\limits_{\Sigma_t}|A|(\cdot,t)\leq C<\infty.$$
     In particular, the flow can be extended past time $T$.
\end{theo}

Combining the $L^\infty$ estimate and the blowup estimate of $A$ in \cite{[H]}, we also obtain the following estimate of the blowup rate of $HA$ at the first finite singularity. This result generalizes Theorem 1.2 of \cite{[LS3]} and Theorem 5.1 of \cite{[C]} and can be seen as another version of Theorem 1.1 of \cite{[W2]} and Theorem 2 of \cite{[KMW]}.
\begin{theo}(Theorem \ref{theo:HA-blowup})
    Let $\x:\Sigma^n\times[0,T)\rightarrow \mathbb{R}^{n+1}$ be a complete smooth mean curvature flow with a finite maximal time T. Each time slice $\Sigma_t$ has bounded second fundamental form. Then there exists a positive constant $\epsilon=\epsilon(n)$ such that
    $$\limsup_{t\to T}\Big(|T-t|\sup_{\Sigma_t}|HA|\Big)\geq\epsilon.$$
\end{theo}

The organization of this paper is as follows. In Sect.2 we recall some basic results on mean curvature flow. In Sect.3 we develop $L^p$ estimate in terms of initial data and $|HA|$ bound, following the argument in \cite{[KMW]}. In Sect.4 we establish the $L^\infty$ estimate by Moser iteration as in \cite{[LS2]} and then derive the extension theorem. In Sect.5 we estimate the blowup rate of $HA$, using the $L^{\infty}$ estimate and the blowup estimate of $A$.

{\bf Acknowledgements}: The author would like to thank H.Z.Li for insightful discussions.

\section{Preliminaries}
Let $\x(p,t):\Sigma^n\to\mathbb{R}^{n+1}$ be a family of smooth immersions. $\{(\Sigma^n,\x(\cdot,t)),0\leq t< T\}$ is called a mean curvature flow if $\x$ satisfies
\beqn
\partial_{t}\x=-H\n,\quad\forall\,t\in[0,T),\label{eq0}
\eeqn
where $A=(h_{ij})$ denotes the second fundamental form and $H=g^{ij}h_{ij}$ denotes the mean curvature. $\Sigma_t:=\x_{t}(\Sigma)$ denotes the time slice of the flow for $t\in[0,T)$.

Some equations are listed here for later calculations. See \cite{[H]} or \cite{[M]} for details.
\begin{lem}\label{lem:equalities for calculation}(Sect.3 of \cite{[H]})
    Along the mean curvature flow,
    \begin{eqnarray}
    &&\partial_{t} d\mu=-H^2 d\mu,\nonumber\\
    &&\partial_{t}|A|^2=2(\nabla^{2}H\cdot A+ H|A^3|),\label{eqn1}\\
    &&2|\nabla H|^2=(\Delta-\partial_{t})H^2+2H^2|A|^2,\label{eqn2}\\
    &&2|\nabla A|^2=(\Delta-\partial_{t})|A|^2+2|A|^4.\label{eqn3}
\end{eqnarray}
where
$$\nabla^2H\cdot A=H_{ij}A_{ij},\quad |A^3|=A_{ij}A_{jk}A_{ki}.$$
\end{lem}
Observe that the equation (\ref{eq0}) is invariant under the rescaling 
$$\tilde{\x}(p,t)=\lambda \Big(\x(p,t_0+\frac{t}{\lambda^2})-x_0\Big),\quad\forall\,(p,t)\in \Sigma\times[-\lambda^2t_0,\lambda^2T-\lambda^2t_0],$$
for some $(x_0,t_0)\in \mathbb{R}^{n+1}\times[0,T)$ and $\lambda>0$. Under the rescaling we have
$$\tilde{A}=\lambda A,\quad \tilde{g}=\lambda^2 g,\quad |\tilde A|=\lambda^{-1}|A|,\quad \tilde{H}=\lambda^{-1}H,$$
where $\tilde{A}$, $\tilde{g}$ and $\tilde{H}$ denote the second fundamental form, the induced metric and the mean curvature of the hypersurface $\tilde{\x}(\Sigma)$ respectively. Moreover, we have
\begin{lem}
Under the rescaling,
$$\mbox{Vol}_{\tilde{g}}\big(B(x_0,\tilde{r})\cap \tilde{\Sigma}_{\tilde{t}}\big)= \lambda^{n} \mbox{Vol}_{g}\big(B(x_0,r)\cap \Sigma_t)\big),$$
$$\int_{B(x_0,\tilde{r})\cap \tilde{\Sigma}_{\tilde{t}}}|\tilde{A}|^p d\tilde{\mu}=\lambda^{n-p} \int_{B(x_0,r)\cap \Sigma_t}|A|^p d\mu,$$
$$\int_{\tilde{t_1}}^{\tilde{t_2}}\int_{B(x_0,\tilde{r})\cap \tilde{M}_{\tilde{t}}}|\tilde{A}|^p d\tilde{\mu}d\tilde{t}
=\lambda^{n+2-p} \int_{t_1}^{t_2}\int_{B(x_0,r)\cap \Sigma_t}|A|^p d\mu dt,$$
where $\tilde{r}=\lambda r$ and $\tilde{t}=t_0+t/\lambda^2$. 
\end{lem}

The standard Moser iteration in Sect.4 depends on the Michael-Simon inequality. See \cite{[H],[MS]}.
\begin{lem}\label{lem:MS Sobolev}(Lemma 5.7 of \cite{[H]})
For any Lipschitz function $f$ with compact support on a hypersurface $\Sigma^n\subset\mathbb{R}^{n+1}$, 
\beqn
\big(\int_{\Sigma}|f|^{\frac{n}{n-1}}d\mu\big)^{\frac{n-1}{n}}\leq c_n \int_{\Sigma}(|\nabla f|+H|f|)d\mu.\label{ineq:MS Sobolev}
\eeqn
\end{lem}

By maximum principle the second fundamental form blows up at least at a rate of $1/2$, which holds for noncompat cases as well.
\begin{lem}\label{lem:A blowup}(Proposition 2.4.6 of \cite{[M]})
Suppose each time slice $\Sigma_t$ has bounded second fundamtental form. Then 
$$\sup_{\Sigma_t}|A|\geq \frac{1}{\sqrt{2(T-t)}}.$$
\end{lem}

\section{$L^p$ estimate}

Throughout this paper $c$ denotes nonnegative constants depending only on $n$ and $p$ while $c_n$ denotes nonnegative constants depending only on $n$, which may change from line to line.

\begin{theo}\label{theo:global Lp est}
    Let $\x:\Sigma^n\times[t_0,t_1]\to\mathbb{R}^{n+1}$ be a closed smooth mean curvature flow satisfying the uniform bound 
    \begin{eqnarray}
        \sup_{\Sigma_t}|HA|(\cdot,t)\leq K(t),\quad \forall\,t\in[t_0,t_1].\label{AH bound}
    \end{eqnarray}
    Then for any $p\geq 4$ there exist positive constants $c=c(n,p)$ such that for any $t\in[t_0,t_1]$,
    $$\int_{\Sigma_t}|A|^p
    \leq \Big(\int_{\Sigma_{t_0}}|A|^p(t_0)+cK(t_0)\int_{\Sigma_{t_0}}|A|^{p-2}(t_0)\Big)e^{\int^{t}_{t_0}(|K'|/K+cK)}.$$
\end{theo}
\begin{proof}
    
    %Subscripts of integration are dropped without causing confusion.
Using equation (\ref{eqn1}) from Lemma \ref{lem:equalities for calculation}, we get
\beqn
\partial_{t}\int_{\Sigma_t}|A|^{p}&=&p\int_{\Sigma_t}|A|^{p-2}(\nabla^2H A+H|A^3|)-\int_{\Sigma_t}|A|^{p}H^2\nonumber\\
&\leq& p(p-1)\int_{\Sigma_t}|A|^{p-2}|\nabla H||\nabla A|+p\int_{\Sigma_t}|H||A|^{p+1}\nonumber\\
&\leq& \frac{1}{4K}\int_{\Sigma_t}|A|^{p}|\nabla H|^2+ p^4K\int_{\Sigma_t}|A|^{p-4}|\nabla A|^2+pK\int_{\Sigma_t}|A|^{p}.\label{est1}
\eeqn
Using equation (\ref{eqn2}), we get 
\beqs
&&2\int_{\Sigma_t}|A|^{p}|\nabla H|^2=\int_{\Sigma_t}|A|^{p}(\Delta-\partial_{t})H^2+2\int_{\Sigma_t}H^2|A|^{p+2}\\
&\leq& 2p\int_{\Sigma_t}|H||A|^{p-1}|\nabla H||\nabla A|-\partial_{t}\int_{\Sigma_t}H^2|A|^p +\int_{\Sigma_t}H^2\partial_t(|A|^p)\\
&&-\int_{\Sigma_t}H^4|A|^p+2\int_{\Sigma_t}H^2|A|^{p+2}.
\eeqs
Since 
$$
2p\int_{\Sigma_t}|H||A|^{p-1}|\nabla H||\nabla A|\leq \int_{\Sigma_t}|A|^p|\nabla H|^2+p^2\int_{\Sigma_t}H^2|A|^{p-2}|\nabla A|^2,
$$
$$
\partial_{t}(|A|^p)=p|A|^{p-2}\big(\nabla^2HA+H|A^3|\big),
$$
$$
p|H|^3|A|^{p+1}\leq H^4|A|^p+\frac{p^2}4H^2|A|^{p+2},
$$
we have
\beqn
\int_{\Sigma_t}|A|^p|\nabla H|^2 
&\leq& p^2\int_{\Sigma_t}H^2|A|^{p-2}|\nabla A|^2-\partial_{t}\int_{\Sigma_t}H^2|A|^p+p\int_{\Sigma_t}H^2|A|^{p-2}\nabla^2HA\nonumber\\
&& +p\int_{\Sigma_t}|H|^3|A|^{p+1}-\int_{\Sigma_t}H^4|A|^p+2\int_{\Sigma_t}H^2|A|^{p+2}\nonumber\\
&\leq& -\partial_{t}\int_{\Sigma_t}H^2|A|^p+p^2K^2\int_{\Sigma_t}|A|^{p-4}|\nabla A|^2+p\int_{\Sigma_t}H^2|A|^{p-2}\nabla^2HA\nonumber\\
&& +\frac{p^2}2 K^2\int_{\Sigma_t}|A|^{p}.\label{est2}
\eeqn
For the third term on the right of (\ref{est2}),
\beqn
p\int_{\Sigma_t}H^2|A|^{p-2}\nabla^2HA 
&\leq& p(p-1)\int_{\Sigma_t}H^2|A|^{p-2}|\nabla H||\nabla A|+2p\int_{\Sigma_t}|H||A|^{p-1}|\nabla H|^2\nonumber\\
&\leq&\frac{1}{2}\int_{\Sigma_t}|A|^p|\nabla H|^2+\frac{p^4}2\int_{\Sigma_t}H^4|A|^{p-4}|\nabla A|^2\nonumber\\
&&+\frac{1}{4}\int_{\Sigma_t}|A|^{p}|\nabla H|^2+4p^2\int_{\Sigma_t}H^2|A|^{p-2}|\nabla H|^2\nonumber\\
&\leq& \frac34\int_{\Sigma_t}|A|^p|\nabla H|^2
+(\frac{np^4}{2}+4np^2)\int_{\Sigma_t}H^2|A|^{p-2}|\nabla A|^2\nonumber\\
&\leq& \frac34\int_{\Sigma_t}|A|^p|\nabla H|^2
+np^2(\frac{p^2}2+4) K^2\int_{\Sigma_t}|A|^{p-4}|\nabla A|^2
.\label{est3}
\eeqn
Substituting (\ref{est3}) into (\ref{est2}), we get
\beqs
\frac{1}{4}\int_{\Sigma_t}|A|^p|\nabla H|^2 
&\leq& -\partial_t\int_{\Sigma_t}H^2|A|^p+np^4K^2\int_{\Sigma_t}|A|^{p-4}|\nabla A|^2+\frac{p^2}{2}K^2\int_{\Sigma_t}|A|^p.
\eeqs
Back to (\ref{est1}), we obtain
\beqs
\partial_{t}\int_{\Sigma_t}|A|^p
&\leq& -\frac1K\partial_{t}\int_{\Sigma_t}H^2|A|^p
+(n+1)p^4K\int_{\Sigma_t}|A|^{p-4}|\nabla A|^2
+(\frac{p^2}2+p)K\int_{\Sigma_t}|A|^p.
\eeqs
By equation (\ref{eqn3}) from Lemma \ref{lem:equalities for calculation}, for $p\geq4$,
\beqs
2\int_{\Sigma_t}|A|^{p-4}|\nabla A|^2
&=&\int_{\Sigma_t}|A|^{p-4}(\Delta-\partial_t)|A|^2+2\int_{\Sigma_t}|A|^{p}\\
&=& -2(p-4)\int_{\Sigma_t}|A|^{p-4}|\nabla A|^2-\frac{2}{p-2}\int_{\Sigma_t}\partial_{t}\big(|A|^{p-2}\big)+2\int_{\Sigma_t}|A|^{p},
\eeqs
i.e.,
\beqs
\int_{\Sigma_t}|A|^{p-4}|\nabla A|^2=
-\frac{1}{(p-2)(p-3)}\partial_{t}\int_{\Sigma_t}|A|^{p-2}
+\frac{1}{p-3}\int_{\Sigma_t}|A|^{p}.
\eeqs
Hence we have
\beqn
&&\partial_{t}\int_{\Sigma_t}|A|^p\nonumber\\
&\leq& -\frac1K\partial_{t}\int_{\Sigma_t}H^2|A|^p 
-\frac{(n+1)p^4}{(p-2)(p-3)}K\partial_{t}\int_{\Sigma_t}|A|^{p-2}
+\Big(\frac{(n+1)p^4}{p-3}+\frac{p^2}2+p\Big)K\int_{\Sigma_t}|A|^p\nonumber\\
&\leq& -\frac1{K}\partial_{t}\int_{\Sigma_t}H^2|A|^p-cK\partial_{t}\int_{\Sigma_t}|A|^{p-2}+cK\int_{\Sigma_t}|A|^p
\label{est4}.
\eeqn
If we set
\beqs
U(t)
=\int_{\Sigma_t}|A|^p+\frac1K\int_{\Sigma_t}H^2|A|^p+cK\int_{\Sigma_t}|A|^{p-2},
\eeqs
then actually (\ref{est4}) becomes
\beqs
U'
&\leq& -\frac{K'}{K^2}\int_{\Sigma_t}H^2|A|^p+cK'\int_{\Sigma_t}|A|^{p-2}+cK\int_{\Sigma_t}|A|^p\\
&\leq& (|K'|/K+cK)U,
\eeqs
which implies
\beqs
 U(t)\leq e^{\int^{t}_{t_0}(|K'|/K+cK)}U(t_0).
\eeqs
In particular,
\beqs
\int_{\Sigma_t}|A|^p
\leq \Big(\int_{\Sigma_{t_0}}|A|^p(t_0)+cK(t_0)\int_{\Sigma_{t_0}}|A|^{p-2}(t_0)\Big)e^{\int^{t}_{t_0}(|K'|/K+cK)}.
\eeqs
\end{proof}

\begin{cor}\label{cor:global invariant Lp est}
    Let $\x:\Sigma^n\times[t_0,t_1]\rightarrow \mathbb{R}^{n+1}$ be a closed smooth mean curvature flow satisfying the uniform bound 
    $$
        \sup_{s\in[t_0,t]}\sup_{\Sigma_s}|HA|(\cdot,s)\leq K(t),\quad \forall\,t\in[t_0,t_1].
    $$
    Then for any $p\geq 4$ there exist positive constants $c=c(n,p)$ such that for any $t\in[t_0,t_1]$
    \beqs
    \int_{\Sigma_t}|A|^p
    &\leq& \Big(\int_{\Sigma_{t_0}}|A|^p(t_0)+cK(t)\int_{\Sigma_{t_0}}|A|^{p-2}(t_0)\Big)e^{c(t-t_0)K(t)}\\
    &\leq& \mbox{Vol}_{g(t_0)}(\Sigma_{t_0})\Big((\sup_{\Sigma_{t_0}}|A|)^p+cK(t)(\sup_{\Sigma_{t_0}}|A|)^{p-2}\Big)e^{c(t-t_0)K(t)}.
    \eeqs
\end{cor}
\begin{proof}
    Fix $\tilde{t}\in[t_0,t_1]$ and $\tilde{K}:=K(\tilde{t})$. Consider instead the rescaled flow $\tilde{\x}:\Sigma\times[\tilde{K}t_0,\tilde{K}\tilde{t}]\to\mathbb{R}^{n+1}$ with
    $\tilde{\x}(\cdot,t)=\sqrt{\tilde{K}}\x(\cdot,t/{\tilde{K}})$.
    One sees 
    $$\sup\limits_{\tilde{\Sigma}_t}|\tilde{H}\tilde{A}|\leq 1,\quad\forall\,t\in[\tilde{K}t_0,\tilde{K}\tilde{t}].$$
    Provided the uniform bound $\hat{K}\equiv1,$ 
    applying Theorem \ref{theo:global Lp est} yields
    \beqs
    &&\tilde{K}^{\frac{n-p}{2}}\int_{\Sigma_{\tilde{t}}}|A|^{p}
    =\int_{\tilde{\Sigma}_{\tilde{K}\tilde{t}}}|\tilde{A}|^p\\
    &\leq& \Big(\int_{\tilde{\Sigma}_{\tilde{K}t_0}}|\tilde{A}|^{p}(\tilde{K}t_0)+c\int_{\tilde{\Sigma}_{\tilde{K}t_0}}|\tilde{A}|^{p-2}(\tilde{K}t_0)\Big)e^{c\tilde{K}(\tilde{t}-t_0)}\\
    &=& \Big({\tilde{K}}^{\frac{n-p}{2}}\int_{\Sigma_{t_0}}|A|^p(t_0)+c{\tilde{K}}^{\frac{n+2-p}{2}}\int_{\Sigma_{t_0}}|A|^{p-2}(t_0)\Big)e^{c\tilde{K}(\tilde{t}-t_0)},
    \eeqs
    which implies
    \beqs
    \int_{\Sigma_{\tilde{t}}}|A|^p
    &\leq& \Big(\int_{\Sigma_{t_0}}|A|^p(t_0)+c\tilde{K}\int_{\Sigma_{t_0}}|A|^{p-2}(t_0)\Big)e^{c\tilde{K}(\tilde{t}-t_0)}.
    \eeqs
    Since the choice of $\tilde{t}$ is arbitrary, we actually obtain the result.
\end{proof}

By the use of some cutoff function, Theorem \ref{theo:global Lp est} can be generalized to a local version. 
\begin{theo}\label{theo:local Lp est}
    Fix $x_0\in\mathbb{R}^{n+1}$ and $r>0$. Let $\x:\Sigma^n\times[t_0,t_1]\rightarrow \mathbb{R}^{n+1}$ be a complete noncompact smooth mean curvature flow satisfying the uniform bound 
    $$
        \sup_{B(x_0,r)\cap \Sigma_{t}}|HA|(\cdot,s)\leq K(t),\quad \forall\, t\in[t_0,t_1].
    $$
    Then for any $p\geq 4$ there exist positive constants $c=c(n,p)$ such that for any $t\in[t_0,t_1]$
    \beqs
    \int_{B(x_0,r/2)\cap \Sigma_t}|A|^p
    &\leq& 
    \Big(\int_{B(x_0,r)\cap \Sigma_{t_0}}|A|^p(t_0)
    +cK(t_0)\int_{B(x_0,r)\cap \Sigma_{t_0}}|A|^{p-2}(t_0)+r^{-p}\mbox{Vol}_{g(t_0)}(B_r)\Big)\\
    && \cdot e^{\int_{t_0}^t(|K'|/K+cK)},
    \eeqs
    where $B_r=B(x_0,r+n^{1/4}\int_{t_0}^{t}\sqrt{K})\cap \Sigma_{t_0}$.
\end{theo}
\begin{proof}
Note that $c$ always denotes nonnegative constants depending on $n$ and $p$, which may vary from line to line. Let $\phi(x,t)$ be a smooth function with compact support which will be determined later. Following the steps in the proof of Theorem \ref{theo:global Lp est}, we have
\beqs
\partial_{t}\int_{\Sigma_t}|A|^p\phi
&\leq& \frac{c}{K}\int_{\Sigma_t}|A|^p|\nabla H|^2\phi+cK\int_{\Sigma_t}|A|^{p-4}|\nabla A|^2\phi+cK\int_{\Sigma_t}|A|^{p}\phi\\
&&+c\int_{\Sigma_t}|A|^{p-1}|\nabla H||\nabla\phi|+\int_{\Sigma_t}|A|^p \partial_{t}\phi\\
&\leq& \frac{c}{K}\int_{\Sigma_t}|A|^p|\nabla H|^2\phi+cK\int_{\Sigma_t}|A|^{p-4}|\nabla A|^2\phi+cK\int_{\Sigma_t}|A|^{p}\phi\\
&&+cK\int_{\Sigma_t}|A|^{p-2}\phi^{-1}|\nabla\phi|^2+\int_{\Sigma_t}|A|^p \partial_{t}\phi,
\eeqs

\beqs
\int_{\Sigma_t}|A|^{p}|\nabla H|^2\phi
&\leq& -\partial_{t}\int_{\Sigma_t}H^2|A|^{p}\phi+c\int_{\Sigma_t}H^2|A|^{p-2}|\nabla A|^2\phi+c\int_{\Sigma_t}H^2|A|^{p+2}\phi\\
&&+2\int_{\Sigma_t}|H||A|^p|\nabla H||\nabla \phi|+p\int_{\Sigma_t}H^2|A|^{p-1}|\nabla H||\nabla\phi|+\int_{\Sigma_t}H^2|A|^p\partial_{t}\phi\\
&\leq& -\partial_{t}\int_{\Sigma_t}H^2|A|^{p}\phi+c\int_{\Sigma_t}H^2|A|^{p-2}|\nabla A|^2\phi+c\int_{\Sigma_t}H^2|A|^{p+2}\phi\\
&&+\frac{1}{2}\int_{\Sigma_t}|A|^p|\nabla H|^2\phi+c\int_{\Sigma_t}H^2|A|^{p}\phi^{-1}|\nabla\phi|^2+\int_{\Sigma_t}H^2|A|^p\partial_{t}\phi\\
&\leq& -\partial_{t}\int_{\Sigma_t}H^2|A|^{p}\phi+cK^2\int_{\Sigma_t}|A|^{p-4}|\nabla A|^2\phi+cK^2\int_{\Sigma_t}|A|^{p}\phi\\
&&+\frac{1}{2}\int_{\Sigma_t}|A|^p|\nabla H|^2\phi+cK^2\int_{\Sigma_t}|A|^{p-2}\phi^{-1}|\nabla\phi|^2+\int_{\Sigma_t}H^2|A|^p\partial_{t}\phi,
\eeqs

\beqs
\frac{c}{K}\int_{\Sigma_t}|A|^p|\nabla H|^2\phi
&\leq& -\frac{c}{K}\partial_{t}\int_{\Sigma_t}H^2|A|^p\phi+cK\int_{\Sigma_t}|A|^{p-4}|\nabla A|^2\phi
+cK\int_{\Sigma_t}|A|^p\phi\\
&&+cK\int_{\Sigma_t}|A|^{p-2}\phi^{-1}|\nabla\phi|^2+\frac{c}{K}\int_{\Sigma_t}H^2|A|^p\partial_{t}\phi,
\eeqs

\beqs
2\int_{\Sigma_t}|A|^{p-4}|\nabla A|^2\phi
&\leq& -c\partial_{t}\int_{\Sigma_t}|A|^{p-2}\phi+2\int_{\Sigma_t}|A|^{p}\phi+2\int_{\Sigma_t}|A|^{p-3}|\nabla A||\nabla\phi|\\
&&+c\int_{\Sigma_t}|A|^{p-2}\partial_{t}\phi\\
&\leq& -c\partial_{t}\int_{\Sigma_t}|A|^{p-2}\phi+2\int_{\Sigma_t}|A|^{p}\phi+\int_{\Sigma_t}|A|^{p-4}|\nabla A|^2\phi\\
&&+c\int_{\Sigma_t}|A|^{p-2}\phi^{-1}|\nabla\phi|^2+c\int_{\Sigma_t}|A|^{p-2}\partial_{t}\phi,
\eeqs

\beqs
K\int_{\Sigma_t}|A|^{p-4}|\nabla A|^2\phi
&\leq& -cK\partial_{t}\int_{\Sigma_t}|A|^{p-2}\phi
+2K\int_{\Sigma_t}|A|^{p}\phi
+cK\int_{\Sigma_t}|A|^{p-2}\phi^{-1}|\nabla\phi|^2\\
&&+cK\int_{\Sigma_t}|A|^{p-2}\partial_{t}\phi,
\eeqs

\beqn
\partial_{t}\int_{\Sigma_t}|A|^{p}\phi
&\leq& -\frac{c}{K}\partial_{t}\int_{\Sigma_t}H^2|A|^p\phi
+cK\int_{\Sigma_t}|A|^{p-4}|\nabla A|^2\phi
+cK\int_{\Sigma_t}|A|^p\phi\nonumber\\
&&+cK\int_{\Sigma_t}|A|^{p-2}\phi^{-1}|\nabla\phi|^2
+\frac{c}{K}\int_{\Sigma_t}H^2|A|^p\partial_{t}\phi+\int_{\Sigma_t}|A|^p\partial_{t}\phi\nonumber\\
&\leq& -\frac{c}K\partial_{t}\int_{\Sigma_t}H^2|A|^{p}\phi
-cK\partial_{t}\int_{\Sigma_t}|A|^{p-2}\phi
+cK\int_{\Sigma_t}|A|^p\phi\nonumber\\
&&+cK\int_{\Sigma_t}|A|^{p-2}\phi^{-1}|\nabla\phi|^2
+\int_{\Sigma_t}|A|^p|\partial_{t}\phi| 
+cK\int_{\Sigma_t}|A|^{p-2}|\partial_{t}\phi|.\label{est5}
\eeqn

Consider a smooth decreasing function $\eta$, which equals $1$ on $[0,r/2]$ and vanishes on $[r,\infty)$, satisfying 
$|\eta'|\leq 3/r$. For any $0<\delta<1$ we set $\psi:=\eta^{1/\delta}$ such that
$$|\psi'|\leq \frac{3}{\delta r}\psi^{1-\delta}.$$
Now we choose $\phi:=\psi(|x-x_0|)$. Then 
\beqs
\phi^{-1}|\nabla\phi|^2\leq n\psi^{-1}(\psi')^2\leq \frac{9n}{\delta^2r^2}\phi^{1-2\delta},
\eeqs
\beqs
|\partial_{t}\phi|=|\psi'||\partial_{t}|x-x_0||\leq |\psi'||H|\leq \frac{3}{\delta r}|H|\phi^{1-\delta}.
\eeqs
Take $\delta=\frac1p$. By Young's inequality,
\beqs
\int_{\Sigma_t}|A|^{p-2}\phi^{-1}|\nabla\phi|^2
&\leq& \frac{c}{\delta^2r^2}\int_{\Sigma_t}|A|^{p-2}\phi^{1-2\delta}
\leq c\int_{\Sigma_t}|A|^p\phi+\frac{c}{\delta^pr^p}\int_{\Sigma_t}\phi^{1-p\delta}\\
&\leq& c\int_{\Sigma_t}|A|^p\phi+\frac{c}{r^p}\mbox{Vol}_{g(t)}\big(B(x_0,r)\cap \Sigma_{t}\big),
\eeqs

\beqs
\int_{\Sigma_t}|A|^{p-2}|\partial_{t}\phi|
&\leq& \frac{c}{\delta r}\int_{\Sigma_t}|A|^{p-1}\phi^{1-\delta}
\leq c\int_{\Sigma_t}|A|^p\phi+\frac{c}{\delta^pr^p}\int_{\Sigma_t}\phi^{1-p\delta}\\
&\leq& c\int_{\Sigma_t}|A|^p\phi+\frac{c}{r^p}\mbox{Vol}_{g(t)}\big(B(x_0,r)\cap \Sigma_{t}\big),
\eeqs

\beqs
\int_{\Sigma_t}|A|^{p}|\partial_{t}\phi|
&\leq& \frac{c}{\delta r}\int_{\Sigma_t}|H||A|^{p}\phi^{1-\delta}
\leq \frac{c}{\delta r}K\int_{\Sigma_t}|A|^{p-1}\phi^{1-\delta}\\
&\leq& cK\int_{\Sigma_t}|A|^p\phi+\frac{c}{r^p}K\mbox{Vol}_{g(t)}\big(B(x_0,r)\cap \Sigma_{t}\big).
\eeqs
Back to (\ref{est5}), we obtain
\beqn
&&\partial_{t}\int_{\Sigma_t}|A|^p\phi+\frac{c}K\partial_{t}\int_{\Sigma_t}H^2|A|^p\phi+cK\partial_{t}\int_{\Sigma_t}|A|^{p-2}\phi\nonumber\\
&\leq& cK\int_{\Sigma_t}|A|^p\phi 
+ \frac{cK}{r^p}\mbox{Vol}_{g(t)}\big(B(x_0,r)\cap \Sigma_{t}\big).\label{est6}
\eeqn
If we set
$$U(t)=\int_{\Sigma_t}|A|^p\phi+\frac{c}K\int_{\Sigma_t}H^2|A|^p\phi+cK\int_{\Sigma_t}|A|^{p-2}\phi,$$
then actually it becomes
\beqs
U'
&\leq& -\frac{cK'}{K^2}\int_{\Sigma_t}H^2|A|^p\phi
+cK'\int_{\Sigma_t}|A|^{p-2}\phi
+cK\int_{\Sigma_t}|A|^p\phi
+\frac{cK}{r^p}\mbox{Vol}_{g(t)}\big(B(x_0,r)\cap \Sigma_{t}\big)\\
&\leq& (|K'|/K+cK)U
+cr^{-p}K\mbox{Vol}_{g(t)}\big(B(x_0,r)\cap \Sigma_{t}\big).
\eeqs
Since 
$$\partial_{t}|x-x_0|\leq|H|\leq n^{1/4}\sqrt{K}$$
and
$$\partial_{t}d\mu=-H^2d\mu,$$ 
we know
$$B(x_0,r)\cap \Sigma_t\subset B(x_0,r+n^{1/4}\int_{t_0}^{t}\sqrt{K})\cap \Sigma_{t_0}:=B_r,$$
$$\mbox{Vol}_{g(s)}\big(B(x_0,r)\cap \Sigma_{s}\big)\leq \mbox{Vol}_{g(t_0)}(B_r),\quad\forall\,s\in[t_0,t].$$
Then for any $s\in[t_0,t]$,
\beqs
\partial_{s}\Big(e^{-\int_{t_0}^s(|K'|/K+cK)}U(s)\Big)
&\leq& cr^{-p}Ke^{-\int_{t_0}^s(|K'|/K|+cK)}\mbox{Vol}_{g(t_0)}(B_r)\\
&\leq& r^{-p}(cK+|K'|/K)e^{-\int_{t_0}^s(|K'|/K|+cK)}\mbox{Vol}_{g(t_0)}(B_r)\\
&=&\partial_{s}\Big(-e^{-\int_{t_0}^s(|K'|/K+cK)}\Big)r^{-p}\mbox{Vol}_{g(t_0)}(B_r),
\eeqs
i.e.,
\beqs
\partial_{s}\Big(e^{-\int_{t_0}^s(|K'|/K+cK)}\big(U(s)+r^{-p}\mbox{Vol}_{g(t_0)}(B_r)\big)\Big)\leq 0,
\eeqs
which implies
\beqn
U(s)\leq e^{\int_{t_0}^s(|K'|/K+cK)}\Big(U(t_0)+r^{-p}\mbox{Vol}_{g(t_0)}(B_r)\Big),\quad\forall\,s\in[t_0,t].\label{est7}
\eeqn
In particular,
\beqs
\int_{B(x_0,r/2)\cap \Sigma_t}|A|^p
&\leq& 
\Big(\int_{B(x_0,r)\cap \Sigma_{t_0}}|A|^p(t_0)
+cK(t_0)\int_{B(x_0,r)\cap \Sigma_{t_0}}|A|^{p-2}(t_0)+r^{-p}\mbox{Vol}_{g(t_0)}(B_r)\Big)\\
&&\cdot e^{\int_{t_0}^t(|K'|/K+cK)}.
\eeqs
\end{proof}

\begin{cor}\label{cor:local invariant Lp est}
Fix $x_0\in\mathbb{R}^{n+1}$ and $r>0$. Let $\x:\Sigma^n\times[t_0,t_1]\rightarrow\mathbb{R}^{n+1}$ be a complete noncompact smooth mean curvature flow satisfying the uniform bound
$$\sup_{s\in[t_0,t]}\sup_{B(x_0,r)\cap \Sigma_s}|HA|(\cdot,s)\leq K(t),\quad\forall\,t\in[t_0,t_1].$$
Then for any $p\geq4$ there exist positive constants $c=c(n,p)$ such that for any $t\in[t_0,t_1]$
\beqs
\int_{B(x_0,r/2)\cap \Sigma_t}|A|^p
&\leq& \Big(\int_{B(x_0,r)\cap \Sigma_{t_0}}|A|^p(t_0)
+cK(t)\int_{B(x_0,r)\cap \Sigma_{t_0}}|A|^{p-2}(t_0)
+r^{-p}\mbox{Vol}_{g(t_0)}(\tilde{B}_r)\Big)\\
&&\cdot e^{c(t-t_0)K(t)},
\eeqs
where $\tilde{B}_r=B(x_0,r+n^{1/4}(t-t_0)\sqrt{K(t)})\cap\Sigma_{t_0}$.
\end{cor}
\begin{proof}
    As in the proof of Theorem \ref{cor:global invariant Lp est}, by rescaling argument and Theorem \ref{theo:global Lp est} one has 
    \beqs
    &&K^{\frac{n-p}{2}}\int_{B(x_0,\frac{r}{2})\cap \Sigma_t}|A|^p
    = \int_{B(x_0,\frac{\sqrt{K}r}{2})\cap\tilde{\Sigma}_{Kt}}|\tilde{A}|^p\\
    &\leq&  \Big(\int_{B(x_0,\sqrt{K}r)\cap\tilde{\Sigma}_{K t_0}}|\tilde{A}|^p(K t_0)
    +c\int_{B(x_0,\sqrt{K}r)\cap\tilde{\Sigma}_{K t_0}}|\tilde{A}|^{p-2}(K t_0)
    +(\sqrt{K}r)^{-p}\mbox{Vol}_{\tilde{g}(K t_0)}(\tilde{B})\Big)e^{c(t-t_0)K}\\
    &\leq& \Big(K^{\frac{n-p}{2}}\int_{B(x_0,r)\cap \Sigma_{t_0}}|A|^p(t_0)
    +cK^{\frac{n+2-p}{2}}\int_{B(x_0,r)\cap \Sigma_{t_0}}|A|^{p-2}(t_0)
    +K^{\frac{n-p}2}r^{-p}\mbox{Vol}_{g(t_0)}(\tilde{B})\Big)e^{c(t-t_0)K},
    \eeqs
    where 
    \beqs
    \tilde{B}
    &=& B(x_0,\sqrt{K}r+n^{1/4}(t-t_0)K)\cap\tilde{\Sigma}_{K t_0}\\
    &=& B(x_0,r+n^{1/4}(t-t_0)\sqrt{K})\cap \Sigma_{t_0}.
    \eeqs
    Thus we obtain the result.
\end{proof}

\section{$L^\infty$ estimate and extension theorem}
From $L^p$ estimate to $L^\infty$ estimate we require the process of Moser iteration using the Michael-Simon inequality, i.e, Lemma \ref{lem:MS Sobolev}, which relies on mean curvature. We conclude the following result from Lemma 5.2 in \cite{[LS2]}.
\begin{lem}[\textbf{Moser iteration}]\label{lem:Moser iteration}(Lemma 5.2 of \cite{[LS2]})
    Let $\x:\Sigma^n\times[t_0,t_1]\rightarrow \mathbb{R}^{n+1}$ be a smooth mean curvature flow. 
    Consider the differential inequality 
    $$(\partial_{t}-\Delta)v\leq fv,\quad v\geq0.$$
    Fix $x_0\in\mathbb{R}^{n+1}$ and $r>0$. For any $q>n+2$ and $\beta\geq2$ there exists a constant $C=C(n,r,t_1-t_0,q,\beta)$ such that for any $t\in[t_0,t_1]$
    \beqs    
    \|v\|_{L^{\infty}(D_{t,r}')}\leq C\|f\|_{L^{q/2}(D_{t,r})}^{\frac{qn^2}{\beta(q-n-2)}}\big(1+\|H\|_{L^{n+2}(D_{t,r})}^{n+2}\big)^{\frac{qn^3}{\beta(n+2)(q-n-2)}}\|v\|_{L^{\beta}(D_{t,r})},
    \eeqs
    where 
    $$D_{t,r}:=\bigcup_{t_0\leq s\leq t}\big(B(x_0,r)\cap \Sigma_s\big),$$
    $$D_{t,r}':=\bigcup_{(t_0+t)/2\leq s\leq t}\big(B(x_0,r/2)\cap \Sigma_s\big).$$
\end{lem}

Combining Theorem \ref{theo:local Lp est} and Lemma \ref{lem:Moser iteration} we obtain the following local estimate.

\begin{theo}[\textbf{$L^\infty$ estimate}]\label{theo:local est}
    Fix $x_0\in\mathbb{R}^{n+1}$ and $r>0$. Let $\x:\Sigma^n\times[t_0,t_1]\rightarrow \mathbb{R}^{n+1}$ be a complete noncompact smooth mean curvature flow satisfying the uniform bound 
    $$\sup_{B(x_0,r)\cap \Sigma_t}|HA|(\cdot,t)\leq K(t),\quad \forall\, t\in[t_0,t_1].$$
    Then for any $q>n+2$ there exist positive constants $C=C(n,r,t_1-t_0,q,K(t_0))$ and $c=c(n,q)$ such that for any $t\in[t_0,t_1]$
    \beqs
    \sup_{D_{t,r}'}|A|
    \leq C\Big(1+\|A\|_{L^{q}(B_{2r})}^{q}\Big)^c\Big(1+\mbox{Vol}_{g(t_0)}(B_{2r})\Big)^c
    \Big(\int_{t_0}^{t}e^{\int_{t_0}^{s}(|K'|/K+cK)}ds\Big)^{\frac1{q}+\frac{n^2}{q-n-2}},
    \eeqs
    where $B_{2r}=B(x_0,2r+n^{1/4}\int_{t_0}^{t}\sqrt{K})\cap \Sigma_{t_0}.$
    In particular, there exist positive constants $C_0=C_0(n,r,t_1-t_0,\Sigma_0,\int_{t_0}^{t_1}\sqrt{K})$ and $c_n$ such that for any $t\in[t_0,t_1]$
    $$\sup_{B(x_0,r/2)\cap \Sigma_t}|A|
    \leq C_0 \Big(\int_{t_0}^{t}e^{\int_{t_0}^{s}(|K'|/K+c_nK)}ds\Big)^{c_n}
    \leq C_0e^{c_{n}\int_{t_0}^{t}(|K'|/K+K)}.$$
\end{theo}
\begin{proof}
Take $\beta=\frac{n+2}{2}$. Applying Lemma \ref{lem:Moser iteration} to 
$$(\partial_{t}-\Delta)|A|^2=-2|\nabla A|^2+2|A|^4\leq2|A|^2\cdot |A|^2$$
yields
\beqn
\sup\limits_{D_{t,r}'}|A|
&\leq& C'\|A\|_{L^{q}(D_{t,r})}^{\frac{qn^2}{\beta(q-n-2)}}\Big(1+\|A\|_{L^{n+2}(D_{t,r})}^{n+2}\Big)^{\frac{qn^3}{2\beta(n+2)(q-n-2)}}\|A\|_{L^{2\beta}(D_{t,r})}\nonumber\\
&\leq& C'\|A\|_{L^{q}(D_{t,r})}^{\frac{2qn^2}{(n+2)(q-n-2)}}\Big(1+\|A\|_{L^{n+2}(D_{t,r})}^{1+\frac{qn^3}{(n+2)(q-n-2)}}\Big)\nonumber\\
&\leq& C'\|A\|_{L^{q}(D_{t,r})}^{\frac{2qn^2}{(n+2)(q-n-2)}}\Big(1+\|A\|_{L^{q}(D_{t,r})}^{1+\frac{qn^3}{(n+2)(q-n-2)}}\mbox{Vol}(D_{t,r})^{\frac{q-n-2}{q(n+2)}+\frac{n^3}{(n+2)^2}}\Big)\nonumber\\
&\leq& C'\|A\|_{L^{q}(D_{t,r})}^{\frac{2qn^2}{(n+2)(q-n-2)}}\Big(1+\|A\|_{L^{q}(D_{t,r})}^{1+\frac{qn^3}{(n+2)(q-n-2)}}\mbox{Vol}_{g(t_0)}(B_{2r})^{\frac{q-n-2}{q(n+2)}+\frac{n^3}{(n+2)^2}}\Big)\nonumber\\
&\leq& C'\Big(1+\|A\|_{L^q(D_{t,r})}^q\Big)^{\frac1q+\frac{n^2}{q-n-2}}
\Big(1+\mbox{Vol}_{g(t_0)}(B_{2r})\Big)^{\frac{q-n-2}{q(n+2)}+\frac{n^3}{(n+2)^2}}
,\label{est9}
\eeqn
where 
$$C'=C'(n,r,t_1-t_0,q),$$
$$\mbox{Vol}(D_{t,r}):=\int_{t_0}^{t}\mbox{Vol}_{g(s)}(B(x_0,r)\cap \Sigma_s)ds,$$
$$B_{2r}=B(x_0,2r+n^{1/4}\int_{t_0}^{t}\sqrt{K})\cap \Sigma_{t_0}
\leq B(x_0,2r+n^{1/4}\int_{t_0}^{t_1}\sqrt{K})\cap \Sigma_{t_0}:=B'_{2r}.$$
It is derived from Theorem \ref{theo:local Lp est} that for $q>n+2$
\beqs
\|A\|_{L^{q}(D_{t,r})}^{q}
&\leq& \Big(\|A\|_{L^{q}(B(x_0,2r)\cap\Sigma_{t_0})}^{q}
+cK(t_0)\|A\|_{L^{q-2}(B(x_0,2r)\cap\Sigma_{t_0})}^{q-2}
+(2r)^{-q}\mbox{Vol}_{g(t_0)}(B_{2r})\Big)\\
&&\cdot\int_{t_0}^{t}e^{\int_{t_0}^{s}(|K'|/K+cK)}ds,
\eeqs
where $c=c(n,q)$. 
Note that
\beqs
&&\|A\|_{L^{q}(B(x_0,2r)\cap\Sigma_{t_0})}^{q}
+cK(t_0)\|A\|_{L^{q-2}(B(x_0,2r)\cap\Sigma_{t_0})}^{q-2}
+(2r)^{-q}\mbox{Vol}_{g(t_0)}(B_{2r})\\
&\leq& \|A\|_{L^{q}(B(x_0,2r)\cap\Sigma_{t_0})}^{q}
+cK(t_0)\|A\|_{L^{q}(B(x_0,2r)\cap\Sigma_{t_0})}^{q-2}\mbox{Vol}_{g(t_0)}(B_{2r})^{\frac2q}
+(2r)^{-q}\mbox{Vol}_{g(t_0)}(B_{2r})\\
&\leq& C(n,r,q,K(t_0))\Big(1+\|A\|_{L^{q}(B(x_0,2r)\cap\Sigma_{t_0})}\Big)^{q}\Big(1+\mbox{Vol}_{g(t_0)}(B_{2r})\Big).
\eeqs
The final coefficient is 
\beqs
    &&C\Big(1+\|A\|_{L^{q}(B(x_0,2r)\cap\Sigma_0)}\Big)^{q(\frac1q+\frac{n^2}{q-n-2})}\Big(1+\mbox{Vol}_{g(t_0)}(B_{2r})\Big)^{\frac1q+\frac{n^2}{q-n-2}+\frac{q-n-2}{q(n+2)}+\frac{n^3}{(n+2)^2}}\\
    &\leq& C\Big(1+\|A\|_{L^{q}(B(x_0,2r)\cap\Sigma_{t_0})}^{q}\Big)^{1+\frac{qn^2}{q-n-2}}\Big(1+\mbox{Vol}_{g(t_0)}(B_{2r})\Big)^{\frac{1}{n+2}+\frac{n^2}{q-n-2}+\frac{n^3}{(n+2)^2}},
\eeqs
where $C=C(n,r,t_1-t_0,q,K(t_0))$.
Back to (\ref{est9}), we have the local $L^{\infty}$ estimate
\beqs
    \sup_{D_{t,r}'}|A|
    \leq C\Big(1+\|A\|_{L^{q}(B(x_0,2r)\cap\Sigma_{t_0})}^{q}\Big)^c\Big(1+\mbox{Vol}_{g(t_0)}(B_{2r})\Big)^c
    \Big(\int_{t_0}^{t}e^{\int_{t_0}^{s}(|K'|/K+cK)}ds\Big)^{\frac1{q}+\frac{n^2}{q-n-2}}.
\eeqs
The second result follows from taking some $q=q(n)>n+2$.
\end{proof}

As an application of the local estimate above, one immediately gets the following extension theorem about $HA$. 
\begin{cor}[\textbf{HA-extension}]\label{cor:AH-extension}
    Let $\x:\Sigma^n\times[0,T)\rightarrow \mathbb{R}^{n+1}$ be a complete noncompact smooth mean curvature flow. Each time slice $\Sigma_t$ has bounded $HA$.
     There exists a positive constant $C=C(n,T,K,V,E,q)$ such that if
     \begin{enumerate}
\item[(1)] the $|HA|$ bound satisfies $$\sup_{t\in[0,T)}\sup_{\Sigma_t}|HA|(\cdot,t)\leq K<\infty;$$  
\item[(2)] the initial data satisfies a uniform volume bound 
\beqs
\sup_{x\in\Sigma_0}\mbox{Vol}_{g(0)}(B(x,1+n^{1/4}T\sqrt{K})\cap\Sigma_0)\leq V<\infty;
\eeqs
\item[(3)] the initial data satisfies an integral bound 
\beqs
\sup_{x\in\Sigma_0}\|A\|_{L^{q}(B(x,1)\cap\Sigma_0)}\leq E<\infty
\eeqs
for some $q>n+2$,
     \end{enumerate}
     then
     $$\limsup_{t\to T}\sup\limits_{\Sigma_t}|A|(\cdot,t)\leq C<\infty.$$
     In particular, the flow can be extended past time $T$.
\end{cor}
\begin{proof}
It suffices to take $r=1$ in Theorem \ref{theo:local est}.
\end{proof}

\section{Blowup estimate of $HA$}

In this section we derive a blowup estiamte of $HA$ from Theorem \ref{theo:local est} and Lemma \ref{lem:A blowup}, which also implies a blowup estimate of mean curvature.

\begin{theo}[\textbf{HA-blowup}]\label{theo:HA-blowup}
Let $\x:\Sigma^n\times[0,T)\rightarrow \mathbb{R}^{n+1}$ be a complete noncompact smooth mean curvature flow with a finite maximal time T. Each time slice $\Sigma_t$ has bounded second fundamental form. Then there exists a positive constant $\epsilon=\epsilon(n)$ such that
$$\limsup_{t\to T}\Big(|T-t|\sup_{\Sigma_t}|HA|\Big)\geq\epsilon.$$
\end{theo}
\begin{proof}
For otherwise one finds $t_0\in[0,T)$ and $\epsilon>0$ such that 
$$\sup\limits_{\Sigma_t}|HA|< \frac{\epsilon}{T-t},\quad\forall\,t\in[t_0,T).$$
Actually we find a smooth mean curvature flow $\x:\Sigma^n\times [t_0,T)\rightarrow \mathbb{R}^{n+1}$ with a $|HA|$ bound 
$$K(t)=\frac{\epsilon}{T-t}.$$
For $t$ close to $T$,
\beqs
\int_{t_0}^{t}(|K'|/K+c_nK)
=\int_{t_0}^{t}\frac{1+c_n\epsilon}{T-s}ds=(1+c_n\epsilon)\log(\frac{T-t_0}{T-t}),
\eeqs
\beqs
\int_{t_0}^{t}e^{\int_{t_0}^{s}(|K'|/K+c_nK)}ds
=\int_{t_0}^{t}(\frac{T-t_0}{T-s})^{1+c_n\epsilon}ds
=\frac{(T-t_0)^{1+c_n\epsilon}}{c_n\epsilon}\Big((T-t)^{-c_n\epsilon}-(T-t)^{-c_n\epsilon}\Big).
\eeqs
On the other hand, by Lemma \ref{lem:A blowup} we know
$$\sup\limits_{\Sigma_t}|A|\geq \frac{1}{\sqrt{2}}(T-t)^{-\frac{1}{2}}.$$
Note that $\Sigma_{t_0}$ has bounded geometry. Combining the inequalities and applying Theorem \ref{theo:local est} imply that as $t\to T$,
$$(T-t)^{-\frac{1}{2}}\leq C(n,T,\Sigma_{t_0},\epsilon)(T-t)^{-c_n\epsilon}.$$
Letting $c_n\epsilon<\frac12$ causes a contradiction.
\end{proof}

Remark that Theorem \ref{theo:HA-blowup} certainly works for the closed cases. The type-I blowup is optimal since the standard sphere $S^n\hookrightarrow\mathbb{R}^{n+1}$ satisfies $|HA|=\frac{n}{2(T-t)}$.

\begin{cor}[\textbf{H-blowup}]\label{cor:H-blowup}
    Let $\x:\Sigma^n\times[0,T)\rightarrow \mathbb{R}^{n+1}$ be a complete noncompact smooth mean curvature flow with a finite maximal time T. Each time slice $\Sigma_t$ has bounded second fundamental form. Assume that
    $$\limsup_{t\to T}\Big(|T-t|^{\lambda}\sup_{\Sigma_t}|A|\Big)<\infty$$
    for some $\lambda\in[\frac12,1).$ Then we have the blowup estimate of mean curvature
    $$\limsup_{t\to T}\Big(|T-t|^{1-\lambda}\sup_{\Sigma_t}|H|\Big)>0.$$
\end{cor}
\begin{proof}
For otherwise for any $\varepsilon>0$ one finds $t_{\varepsilon}$ such that 
$$\sup_{\Sigma_t}|H|\leq\varepsilon(T-t)^{\lambda-1},\quad\forall\,t\in[t_{\varepsilon},T).$$
By the assumption there exist nonnegative constants $t_1\in[0,T)$ and 
$$C:=\limsup_{t\to T}\Big(|T-t|^{\lambda}\sup_{\Sigma_t}|A|\Big)<\infty$$ such that
$$\sup_{\Sigma_t}|A|\leq C(T-t)^{-\lambda},\quad\forall\,t\in[t_{1},T).$$
Hence we have
$$\sup_{\Sigma_t}|HA|\leq C\varepsilon(T-t)^{-1},\quad\forall\,t\in[\max\{t_{\varepsilon},t_1\},T).$$
Recall the constant $\epsilon=\epsilon(n)$ in Theorem \ref{theo:HA-blowup}. Letting $C\varepsilon<\epsilon$ contradicts Theorem \ref{theo:HA-blowup}.
\end{proof}

Remark that by Theorem 5.1 of \cite{[C]} Cooper proved the blowup of mean curvature under the same assumption in Corollary \ref{cor:H-blowup} and by Theorem 1.2 of \cite{[LS3]} Le-Sesum proved the case of $\lambda=\frac12$. Hence Theorem \ref{theo:HA-blowup} and Corollary \ref{cor:H-blowup} can be seen as generalizations of these results.

\begin{cor}\label{cor:A-blowup}
    Let $\x:\Sigma^n\times[0,T)\rightarrow \mathbb{R}^{n+1}$ be a complete noncompact smooth mean curvature flow with a finite maximal time T. Each time slice $\Sigma_t$ has bounded second fundamental form. Assume that
    $$\limsup_{t\to T}\Big(|T-t|^{\lambda}\sup_{\Sigma_t}|H|\Big)<\infty$$
    for some $\lambda\in[0,\frac12).$ Then we have the blowup estimate
    $$\limsup_{t\to T}\Big(|T-t|^{1-\lambda}\sup_{\Sigma_t}|A|\Big)>0.$$
    In particular, $t=T$ is a type-II singularity.
\end{cor}
\begin{proof}
By the same argument used in the proof of Corollary \ref{cor:H-blowup} we obtain the result.
\end{proof}


\begin{thebibliography}{10}
    
\bibitem{[CW]} X. X. Chen, B. Wang, \emph{On the conditions to extend Ricci flow(III)}. Int. Math. Res. Not. IMRN 2013, no. 10, 2349-2367.

\bibitem{[C]} A. Cooper, \emph{ A characterization of the singular time of the mean curvature flow}. Proc. Amer. Math. Soc. 139 (2011), no. 8, 2933-2942.

\bibitem{[Ha]}R. S. Hamilton, \emph{ Three-manifolds with positive Ricci curvature}, J. Differential Geom. 17 (1982), no. 2, 255-306.

\bibitem{[H]} G. Huisken, \emph{ Flow by mean curvature of convex surfaces into spheres}, J. Differential Geom. 20 (1984), no. 1, 237–266.

\bibitem{[KMW]} B. Kotschwar, O. Munteanu, J. P. Wang, \emph{ A local curvature estimate for the Ricci flow}, J. Funct. Anal. 271 (2016), no. 9, 2604–2630.

\bibitem{[LW]} H. Z. Li, B. Wang, \emph{ The extension problem of the mean curvature flow(I)}, Invent. Math. 218, 721-777 (2019).

\bibitem{[LS]} N. Q. Le, N. Sesum, \emph{ The mean curvature at the first singular time of the mean curvature flow}. Ann. Inst. H. Poincare Anal. Non Lineaire 27 (2010), no. 6, 1441-1459.

\bibitem{[LS2]} N. Q. Le, N. Sesum, \emph{ On the extension of the mean curvature flow}, Math. Z. 267, 583–604 (2011).

\bibitem{[LS3]} N. Q. Le, N. Sesum, \emph{ Blow-up rate of the mean curvature during the mean curvature flow and a gap theorem for self-shrinkers}, Commun. Anal. Geom. 19(4), 633-659 (2011).

\bibitem{[M]} C. Mantegazza, \emph{ Lecture notes on mean curvature flow}, Progress in Mathematics, 290. Birkhäuser/Springer Basel AG, Basel, 2011. xii+166 pp.

\bibitem{[MW]} O. Munteanu, M. T. Wang, \emph{ The curvature of gradient Ricci solitons}. Math. Res. Lett. 18 (2011), no. 6, 1051–1069. (Reviewer: Bo Yang) 53C21 (53C25).

\bibitem{[MS]} J. H. Michael, L. M. Simon, \emph{ Sobolev and mean value inequalities on generalized submanifolds of $\mathbb{R}^{n+1}$}, Comm. Pure Appl. Math. 26 (1973), 316-379.

\bibitem{[S]} N. Sesum, \emph{ Curvature tensor under the Ricci flow},  Amer. J. Math. 127 (2005), no. 6, 1315-1324.

\bibitem{[W]} B. Wang, \emph{ On the conditions to extend Ricci flow}, Int. Math. Res. Not. IMRN 2008, no. 8, Art. ID rnn012, 30 pp.

\bibitem{[W2]} B. Wang, \emph{ On the conditions to extend Ricci flow(II)}. Int. Math. Res. Not. IMRN 2012, no. 14, 3192-3223.

\bibitem{[XYZ]} H. W. Xu, F. Ye, E. T. Zhao, \emph{ Extend mean curvature flow with finite integral curvature}. Asian J. Math. 15 (2011), no. 4, 549-556.



\end{thebibliography}
\end{document}